\documentclass[11pt]{amsart}
\usepackage{amssymb, amscd}
\usepackage[all]{xy}

\title{Morse-type integrals on non-K\"ahler manifolds}

\author[S. Ko\l odziej]{S\l awomir Ko\l odziej}
\address{Faculty of Mathematics and Computer Science, Jagiellonian University 30-348, Krak\'ow, \L ojasiewicza 6, Poland.}
\email{Slawomir.Kolodziej@im.uj.edu.pl}
\author[V. Tosatti]{Valentino Tosatti}
\address{Department of Mathematics, Northwestern University, 2033 Sheridan Road, Evanston, IL 60208, USA.}
\email{tosatti@math.northwestern.edu}
\dedicatory{Dedicated to Professor D.H. Phong on the occasion of his 65th birthday}

\theoremstyle{plain}
\newtheorem{theorem}{Theorem}[section]
\newtheorem{proposition}[theorem]{Proposition}

\newtheorem{conjecture}[theorem]{Conjecture}

\theoremstyle{definition}
\newtheorem{remark}[theorem]{Remark}
\numberwithin{equation}{section}

\newcommand{\Vol}{\mathrm{Vol}}

\newcommand{\del}{\partial}
\newcommand{\de}{\partial}
\newcommand{\dbar}{\overline{\del}}
\newcommand{\db}{\overline{\del}}

\newcommand{\ddbar}{\sqrt{-1}\del\dbar}

\newcommand{\ve}{\varepsilon}
\newcommand{\vp}{\varphi}
\newcommand{\ti}[1]{\tilde{#1}}

\renewcommand{\leq}{\leqslant}
\renewcommand{\geq}{\geqslant}

\begin{document}
\begin{abstract} We pose a conjecture about Morse-type integrals in nef $(1,1)$ classes on compact Hermitian manifolds, and we show that it holds for semipositive classes, or when the manifold admits certain special Hermitian metrics.
\end{abstract}

\maketitle

\section{Introduction}
Let $(X^n,\omega)$ be a compact Hermitian manifold and $\alpha$ a closed real $(1,1)$ form on $X$. The cohomology class $[\alpha]$ (in Bott-Chern cohomology) consists precisely of all closed real $(1,1)$ forms which can be written in the form
$\alpha+\ddbar u,$
for some $u\in C^\infty(X,\mathbb{R})$.

A class $[\alpha]$ is called {\em nef} if it contains representatives with arbitrary small negative part, namely if for every $\ve>0$ there is $u_\ve\in C^\infty(X,\mathbb{R})$ such that $\alpha+\ddbar u_\ve\geq -\ve\omega$. When $X$ is K\"ahler, this is equivalent to the class $[\alpha]$ being a limit of K\"ahler classes.

A class $[\alpha]$ is called {\em big} if it contains a K\"ahler current $T$ in the following sense: there is $\ve>0$ and there exists a quasi-psh function $u$ on $X$ (locally the sum of psh plus smooth) such that $T:=\alpha+\ddbar u\geq \ve\omega$ holds in the weak sense of currents. In this case Demailly-P\u{a}un show that $X$ must be bimeromorphic to a compact K\"ahler manifold (i.e. $X$ is in Fujiki's class $\mathcal{C}$), and then Boucksom \cite{Bo} defines the {\em volume} of $[\alpha]$ to be
$$\Vol([\alpha])=\sup_{T}\int_XT_{ac}^n>0,$$
where the supremum is over all K\"ahler currents in the class $[\alpha]$, and $T_{ac}$ denotes the absolutely continuous part of $T$ in the Lebesgue decomposition (see \cite{Bo} for more details). On the other hand, if $[\alpha]$ is not big (on general compact complex manifolds), then one can simply define $\Vol([\alpha])=0$. Boucksom also shows that if $X$ is in class $\mathcal{C}$ and $[\alpha]$ is nef, then $$\Vol([\alpha])=\int_X\alpha^n.$$
The same formula is conjectured to hold for general compact complex manifolds, which boils down to a conjecture of Demailly-P\u{a}un \cite{DP} to the effect that a nef class $[\alpha]$ with $\int_X\alpha^n>0$ should be big.

If $X$ is K\"ahler, a different formula for the volume of $[\alpha]$ was proposed by Demailly \cite{De}, inspired by his holomorphic Morse inequalities \cite{De3}:
\begin{conjecture}\label{c}
For $(X^n,\omega)$ compact K\"ahler and $\alpha$ a closed real $(1,1)$ form we have
\begin{equation}\label{goal2}
\Vol([\alpha])=\inf_{u\in C^\infty(X,\mathbb{R})}\int_{X(\alpha+\ddbar u,0)}(\alpha+\ddbar u)^n,
\end{equation}
where $X(\alpha+\ddbar u,0)$ denotes the set of all points $x\in X$ such that $(\alpha+\ddbar u)(x)\geq 0$.
\end{conjecture}

In \cite{De} Demailly shows that the inequality
$$\Vol([\alpha])\leq\inf_{u\in C^\infty(X,\mathbb{R})}\int_{X(\alpha+\ddbar u,0)}(\alpha+\ddbar u)^n,$$
holds, using the regularity results with Berman \cite{BD} (one can instead also use the more recent work of Berman \cite{Be}). In \cite{De2} it is shown that Conjecture \ref{c} holds whenever the orthogonality conjecture of Zariski decompositions of Boucksom-Demailly-P\u{a}un-Peternell \cite{BDPP} holds. In particular, Conjecture \ref{c} holds when $X$ is projective (for $[\alpha]=c_1(L)$ by \cite{BDPP} and for general $[\alpha]$ by Witt Nystr\"om \cite{WN}). It is also not hard to see that Conjecture \ref{c} holds when $[\alpha]$ is nef, see Proposition \ref{1} below.
Nevertheless, Conjecture \ref{c} remains open in full generality, even in the special case when $\Vol([\alpha])=0$.

Our main interest is in a version of Conjecture \ref{c} for nef classes on non-K\"ahler manifolds, which also encompasses a question posed by the second-named author in \cite[Remark 3.2]{To}:

\begin{conjecture}\label{con}For $(X^n,\omega)$ compact Hermitian and $\alpha$ a closed real $(1,1)$ form such that $[\alpha]$ is nef we have
\begin{equation}\label{goal}
\int_X \alpha^n=\inf_{u\in C^\infty(X,\mathbb{R})}\int_{X(\alpha+\ddbar u,0)}(\alpha+\ddbar u)^n.
\end{equation}
\end{conjecture}
As remarked above, in the non-K\"ahler case $\int_X \alpha^n$ is only conjectured to equal $\Vol([\alpha])$ for nef classes, which explains the relation between Conjectures \ref{c} and \ref{con}. Also, \eqref{goal} in particular implies that
$$\int_{X\backslash X(\alpha+\ddbar u,0)}(\alpha+\ddbar u)^n\leq 0,$$
for all $u\in C^\infty(X,\mathbb{R})$, which is an elementary-looking statement reminescent of Siu's ``calculus inequalities'' \cite{Si} derived from Demailly's holomorphic Morse inequalities \cite{De3}. As was observed in \cite{To}, Conjecture \ref{con} also has applications to complex Monge-Amp\`ere equations on non-K\"ahler manifolds, as we shall explain in Section \ref{appl} below.

As mentioned above, and recalled in Proposition \ref{1} below, Conjecture \ref{con} is known to hold when $X$ is K\"ahler, or more generally in class $\mathcal{C}$, and therefore it holds if $[\alpha]$ is also big.

Our main result is the following:

\begin{theorem}\label{main}
Conjecture \ref{con} holds if either
\begin{itemize}
\item[(a)] The class $[\alpha]$ is semipositive, i.e. there is $v\in C^\infty(X,\mathbb{R})$ such that $\alpha+\ddbar v\geq 0$, or
\item[(b)] The manifold $X$ admits a Hermitian metric $\omega$ with $\de\db\omega=0=\de\db(\omega^2)$.
\end{itemize}
In particular, Conjecture \ref{con} holds when $n=2$.
\end{theorem}

To prove this, the idea is to obtain a suitable $L^\infty$ bound for $(\alpha+\ve\omega)$-psh envelope functions, by making use of our assumptions (a) or (b). Here recent regularity results for quasi-psh envelopes are used \cite{Be2,Be,BD,CZ,PS,PS2,PS3,To2}. Once this estimate is obtained, we employ a Chern-Levine-Nirenberg type argument to deduce the main result. Obtaining such a suitable $L^\infty$ bound is the main difficulty in proving Conjecture \ref{con} in general, see Remark \ref{rmk} below.\\

{\bf Acknowledgments. }It is our pleasure to dedicate this work to Professor D.H. Phong on the occasion of his 65th birthday, in honor of his influence in Mathematics. We are also grateful to the referees for useful comments. The first-named author was partially supported by NCN grant
2017/27/B/ST1/01145. The second-named author was partially supported by NSF grant DMS-1610278. This work was completed during the second-named author's visit to the Center for Mathematical Sciences and Applications at Harvard University, which he thanks for the hospitality.

\section{Proof of the main result}
To start, we show the ``easy half'' of \eqref{goal}:
\begin{proposition}\label{easy}
For $(X^n,\omega)$ compact Hermitian and $\alpha$ a closed real $(1,1)$ form such that $[\alpha]$ is nef, if either one of assumptions (a) and (b) of Theorem \ref{main} holds, then we have
\begin{equation}\label{g2}
\int_X \alpha^n\geq\inf_{u\in C^\infty(X,\mathbb{R})}\int_{X(\alpha+\ddbar u,0)}(\alpha+\ddbar u)^n.
\end{equation}
\end{proposition}
\begin{proof}
Assume first that (a) holds, so there is $v\in C^\infty(X,\mathbb{R})$ such that $\alpha+\ddbar v\geq 0$. Then $X(\alpha+\ddbar v,0)=X$ and so
$$\inf_{u\in C^\infty(X,\mathbb{R})}\int_{X(\alpha+\ddbar u,0)}(\alpha+\ddbar u)^n\leq\int_X(\alpha+\ddbar v)^n=\int_X\alpha^n,$$
as desired.

Next assume (b), and fix a Hermitian metric $\omega$ with $\de\db\omega=0=\de\db(\omega^2)$ (which is easily seen to imply $\de\db(\omega^k)=0$ for all $k$).
Given any $\ve>0$ there is $u_\ve\in C^\infty(X,\mathbb{R})$ such that $\alpha+\ve\omega+\ddbar u_\ve\geq 0$ and so $X(\alpha+\ve\omega+\ddbar u_\ve,0)=X$ and
\[\begin{split}\inf_{u\in C^\infty(X,\mathbb{R})}\int_{X(\alpha+\ve\omega+\ddbar u,0)}(\alpha+\ve\omega+\ddbar u)^n&\leq\int_X(\alpha+\ve\omega+\ddbar u_\ve)^n\\
&=\int_X(\alpha+\ve\omega)^n,
\end{split}\]
integrating by parts and using that $\de\db(\omega^k)=0$.
On the other hand, given any $u\in C^\infty(X,\mathbb{R})$ we clearly have
$$X(\alpha+\ddbar u,0)\subset X(\alpha+\ve\omega+\ddbar u,0),$$
and on the set $X(\alpha+\ddbar u,0)$ we have the inequality $(\alpha+\ddbar u)^n \leq (\alpha+\ve\omega+\ddbar u)^n$, and so
\[\begin{split}
\inf_{u\in C^\infty(X,\mathbb{R})}&\int_{X(\alpha+\ddbar u,0)}(\alpha+\ddbar u)^n\\
&\leq\inf_{u\in C^\infty(X,\mathbb{R})}\int_{X(\alpha+\ve\omega+\ddbar u,0)}(\alpha+\ve\omega+\ddbar u)^n\\
&\leq\int_X(\alpha+\ve\omega)^n,
\end{split}\]
and letting $\ve\to 0$ the RHS converges to $\int_X\alpha^n,$
thus proving \eqref{g2}.
\end{proof}

Before proving Theorem \ref{main}, let us recall how \eqref{goal} is proved when $X$ is in class $\mathcal{C}$ (bimeromorphic to K\"ahler), which holds for example when $[\alpha]$ is also big.

\begin{proposition}[Demailly \cite{De,De2}]\label{1}
Conjecture \ref{con} holds if $X$ is in class $\mathcal{C}$.
\end{proposition}
\begin{proof}
Assume first that $X$ is K\"ahler. Thanks to Proposition \ref{easy}, it suffices to show the inequality
\begin{equation}\label{g3}
\int_X \alpha^n\leq\inf_{u\in C^\infty(X,\mathbb{R})}\int_{X(\alpha+\ddbar u,0)}(\alpha+\ddbar u)^n.
\end{equation}
Fix any $u\in C^\infty(X,\mathbb{R})$, write $\beta=\alpha+\ddbar u$, and for $\ve>0$ let
$$u_\ve(x)=\sup\{\vp(x)\ |\ \vp\in PSH(X,\beta+\ve\omega), \vp\leq 0\}.$$
Since the class $[\beta+\ve\omega]$ is K\"ahler, Berman \cite{Be} shows that $u_\ve\in C^{1,\gamma}(X)$ for all $\gamma<1$, and \cite{CZ,To2} (building upon \cite{CTW}) in fact give $C^{1,1}(X)$. It is then easy to show (see e.g. \cite{To2}) using this that
$$\int_X(\beta+\ve\omega)^n=\int_X(\beta+\ve\omega+\ddbar u_\ve)^n=\int_{\{u_\ve=0\}}(\beta+\ve\omega)^n\leq \int_{X(\beta+\ve\omega,0)}(\beta+\ve\omega)^n,$$
where the first equality is integration by parts (using that $\omega$ is K\"ahler), the second one uses the regularity statement above (namely $u_\ve\in C^{1,1}(X)$ implies that $\nabla^2u_\ve$ vanishes $(\beta+\ve\omega+\ddbar u_\ve)^n$-a.e. on the set $\{u_\ve=0\}$, while $(\beta+\ve\omega+\ddbar u_\ve)^n=0$ on $\{u_\ve<0\}$, see e.g. \cite[(1.1)]{To2}) and the final inequality is simple (see \cite[Proposition 3.1 (iii)]{Be2}). Letting $\ve\to 0$ we get \eqref{g3}.

In the general case when $X$ is in class $\mathcal{C}$, there exists a composition of blowups $\mu:\ti{X}\to X$ such that $\ti{X}$ is K\"ahler. Then $[\mu^*\alpha]$ is nef and big, with clearly
$$\int_{X}\alpha^n=\int_{\ti{X}}\mu^*\alpha^n,$$
while Demailly \cite{De2} shows that
\[\begin{split}
\inf_{u\in C^\infty(X,\mathbb{R})}&\int_{X(\alpha+\ddbar u,0)}(\alpha+\ddbar u)^n\\
&=\inf_{\ti{u}\in C^\infty(\ti{X},\mathbb{R})}\int_{\ti{X}(\mu^*\alpha+\ddbar \ti{u},0)}(\mu^*\alpha+\ddbar \ti{u})^n,
\end{split}\]
and so we are reduced to proving Conjecture \ref{con} on $\ti{X}$, which is K\"ahler.
\end{proof}

We can now give the proof of Theorem \ref{main}
\begin{proof}[Proof of Theorem \ref{main}]
Thanks to Proposition \ref{easy} it suffices to show that \eqref{g3} holds. Fix any $u\in C^\infty(X,\mathbb{R})$, write $\beta=\alpha+\ddbar u$.
By definition for every $\ve>0$ we can find a smooth function $h_\ve$ such that $\alpha+\ve\omega+\ddbar h_\ve>0$. We consider the envelope
\[\begin{split}
u_\ve(x)&=\sup\{\vp(x)\ |\ \vp\in PSH(X,\beta+\ve\omega), \vp\leq 0\},\\
&=-u+h_\ve+\sup\{\vp(x)\ |\ \vp\in PSH(X,\alpha+\ve\omega+\ddbar h_\ve), \vp\leq u-h_\ve\}
\end{split}\]
which thanks to \cite{CZ} (see also \cite{CTW,To2}) satisfies $u_\ve\in C^{1,1}(X)$. As in Proposition \ref{1}, this implies that
\begin{equation}\label{interm}
\int_X(\beta+\ve\omega+\ddbar u_\ve)^n=\int_{\{u_\ve=0\}}(\beta+\ve\omega)^n\leq \int_{X(\beta+\ve\omega,0)}(\beta+\ve\omega)^n,
\end{equation}
where the inequality is again simple (see \cite[Proposition 3.1 (iii)]{Be2}) and
for the first equality we again have that $u_\ve\in C^{1,1}(X)$ implies that $\nabla^2u_\ve$ vanishes $(\beta+\ve\omega+\ddbar u_\ve)^n$-a.e. on $\{u_\ve=0\}$ (see e.g. \cite[(1.1)]{To2}, which does not use the K\"ahler condition), while we still have that
\begin{equation}\label{P}
(\beta+\ve\omega+\ddbar u_\ve)^n=0 \ \ \text{ on   } \{u_\ve<0\}
\end{equation}
using the balayage procedure. Indeed, consider the Monge-Amp\`ere equation with the background Hermitian metric $\beta_\ve = \alpha+\ve\omega+\ddbar h_\ve$.
We need to verify that  the function $\varphi _\ve = u_\ve +u -h_\ve$ satisfies
$$ (\beta_\ve +  \ddbar\varphi _\ve )^n =0
$$
on the open set $U = \{u_\ve<0\} .$ For this fix a coordinate ball $B\subset U$ and use  \cite[Theorem 4.2]{KN} to find a continuous function
$\psi _\ve \in PSH(B,\beta _\ve )$
solving
$$
(\beta_\ve +  \ddbar\psi _\ve )^n =0
$$
in $B,$ together with the boundary condition $\psi _\ve = \varphi _\ve$ on the boundary of $B$. By  \cite[Proposition 2.5]{KN1} we have
 $\psi _\ve \geq \varphi _\ve$ in $B$. Therefore, as in the classical Perron method,   one modifies $\varphi _\ve$ on $B$ setting it equal to  $t\psi _\ve + (1-t) \varphi _\ve$ there.
This new function belongs to the second envelope in the definition of $u_\ve$ above for sufficiently small positive $t$ and thus $\psi _\ve = \varphi _\ve$ in $B$.
Therefore \eqref{P} holds.

On the other hand, if we let $\ve\to 0$ then we easily have
\begin{equation}\label{interm2}
\lim_{\ve\to 0}\int_{X(\beta+\ve\omega,0)}(\beta+\ve\omega)^n=\int_{X(\beta,0)}\beta^n,
\end{equation}
so if we show that
\begin{equation}\label{goal3}
\limsup_{\ve\to 0}\int_X(\beta+\ve\omega+\ddbar u_\ve)^n\geq\int_X\alpha^n,
\end{equation}
then \eqref{g3} follows from \eqref{interm}, \eqref{interm2} and \eqref{goal3}.\\

{\bf Claim. }Assume that
\begin{equation}\label{todo}
\ve^{\delta} \|u_\ve\|_{L^\infty(X)}\to 0,  \  \ \ \ \textrm{for\ some\ }\delta <1/(n-2),
\end{equation}
as $\ve\to 0$. Then we have that
\begin{equation}\label{goal4}
\lim_{\ve\to 0}\int_X(\beta+\ve\omega+\ddbar u_\ve)^n=\int_X\alpha^n,
\end{equation}
and so in particular \eqref{goal3} holds.\\

Before proving this claim, let us use it to conclude the proof of Theorem \ref{main}.
First, it is clear that \eqref{goal3} holds under our assumption (b), since in this case we can integrate by parts
$$\int_X(\beta+\ve\omega+\ddbar u_\ve)^n=\int_X(\beta+\ve\omega)^n\to \int_X\beta^n=\int_X\alpha^n.$$
On the other hand under our assumption (a), there is $v\in C^\infty(X,\mathbb{R})$ such that $\alpha _v := \alpha+\ddbar v\geq 0$. Therefore
$$\vp=v-u-\sup_X(v-u),$$
is a competitor for the supremum defining $u_\ve$, and so
$$0\geq u_\ve\geq v-u-\sup_X(v-u),$$
i.e. $$\|u_\ve\|_{L^\infty(X)}\leq C,$$
for $C$ independent of $\ve,$ and so our Claim applies.

Finally, we prove our Claim. Let us introduce the following notation:
$$
\alpha _{\ve}= \beta+\ve\omega+\ddbar u_\ve = \alpha +\ve\omega+\ddbar (u_\ve +u).$$
For fixed $\ve$ write also
$$
 I(j,k) =\int _X   \alpha _{\ve}^j \wedge \alpha ^k \wedge \omega ^{n-j-k}.
$$
Choose a constant $M$ so large that
\begin{equation}
\label{e0}
\begin{aligned}
-M\omega ^3 &\leq \sqrt{-1} \de\omega \wedge\db \omega \leq M\omega ^3 \\
-M \omega ^{k+1} &\leq  \ddbar (\omega ^{k})   \leq M \omega ^{k+1} ,\ \ k=1, 2, ... , n-1,
\end{aligned}
\end{equation}
where the inequalities here mean that the difference is a  positive form.
In what follows we shall need the estimate for $\ddbar ( \alpha _{\ve}^{p}  \wedge \omega ^{q})$ .
Note that $\de \alpha _{\ve} = \ve \de \omega$ and $\db \alpha _{\ve} = \ve \db \omega$. Thus, with the convention that terms with negative exterior powers vanish,
\begin{equation}
\begin{aligned}
&\de\db  ( \alpha _{\ve}^{p}  \wedge \omega ^{q}) =  \alpha _{\ve}^{p}\wedge \de\db (\omega ^q  )  + 2pq \ve \alpha _{\ve}^{p-1} \wedge \omega ^{q-1}
\wedge\de \omega\wedge \db \omega \\
+& p\ve \alpha _{\ve}^{p-2} \wedge \omega ^{q}\wedge (\alpha _{\ve} \wedge \de\db  \omega  +(p-1) \ve  \de\omega\wedge\db\omega ).
\end{aligned}
\end{equation}
Therefore, by  \eqref{e0}
\begin{equation}
\label{e1}
\begin{aligned}
& |\de\db ( \alpha _{\ve}^{p}  \wedge \omega ^{q})| \leq M(  \alpha _{\ve}^{p}  \wedge \omega ^{q+1} +  2pq \ve \alpha _{\ve}^{p-1} \wedge \omega ^{q+2} \\
+&p\ve  \alpha _{\ve}^{p-1}  \wedge \omega ^{q+2}
+p(p-1)\ve ^2  \alpha _{\ve}^{p-2}  \wedge \omega ^{q+3}.
\end{aligned}
\end{equation}
We  estimate integrating by parts
\begin{equation}
\begin{aligned}
&I(j,k) -  I(j-1,k+1) - \ve I(j-1, k)  \\
=& \int _X \ddbar (u_\ve +u)\wedge \alpha _{\ve}^{j-1} \wedge \alpha  ^k \wedge \omega ^{n-j-k} \\
=&\int _X   (u_\ve +u) \ddbar ( \alpha _{\ve}^{j-1} \wedge \alpha  ^k \wedge \omega ^{n-j-k})  \\
=&\int _X   (u_\ve +u)   \alpha  ^k \wedge \ddbar (\alpha _{\ve}^{j-1} \wedge\omega ^{n-j-k}).
\end{aligned}
\end{equation}
For $j=1$ this gives
$$I(1,k) \leq I(0,k+1)+\ve I(0,k) +C_0 ||u_\ve ||_{L^1 (X) } ,
$$
for some uniform constant $C_0$. But $I(0,k)$ is uniformly bounded for all $k$, and since $u_\ve$ satisfies $0\leq \beta+\ve\omega+\ddbar u_\ve\leq C\omega+\ddbar u_\ve$ and $\sup_X u_\ve=0$, it is well-known that these imply a uniform bound for $||u_\ve ||_{L^1 (X) }$ (see e.g. \cite[Proposition 2.1]{DK}). We thus conclude that $I(1,k)\leq \tilde{C}_0$ for all $k$.

Next, we assume that $j>1$ and use \eqref{todo} and \eqref{e1} to obtain

$$
 I(j,k)   \leq  I(j-1,k+1) + C_0 \ve^{-\delta}[ I(j-1,k)+\ve  I(j-2,k)+\ve ^2 I(j-3,k)] .
$$
Since $I(1,k)$ are uniformly bounded one can iterate the above estimate to obtain
$$
I(j,0) \leq C_1  \ve^{-(j-1)\delta}, \ \ \ j= 1, 2, ... , n.
$$
Now, by Stokes' theorem
\begin{equation}
\begin{aligned}
&\int_X (\beta+\ve\omega+\ddbar u_\ve)^n - \int_X\alpha^n \\
 =& \int_X (\alpha+\ve\omega+\ddbar (u_\ve +u))^n - \int_X (\alpha +\ddbar (u_\ve +u) )^n \\
=& \int_X  \ve\omega \wedge \sum_{k=0}^{n-1}   \alpha _\ve ^{n-k-1} \wedge (\alpha_\ve  -\ve \omega  )^k \\
= &\int_X  \ve\omega \wedge \sum_{k=0}^{n-1}  \sum_{s=0}^{k} (-1)^s\binom{k}{s}  \ve ^s \omega ^s  \wedge \alpha _\ve  ^{n-s-1}\\
= & \sum_{k=1}^{n-1}  \sum_{s=0}^{k} (-1)^s\binom{k}{s} \ve ^{s+1}I(n-s-1, 0).
\end{aligned}
\end{equation}
By the above estimates  $I(j,0) \leq C_1  \ve^{-(j-1)\delta}$ we see that the absolute value of the RHS tends to zero as $\ve \to 0$ for $ \delta <1/(n-2)$.
\end{proof}
\begin{remark}\label{rmk}
As shown in the arguments above, to prove the ``half'' \eqref{g3} of Conjecture \ref{con} in general the problem is to show that \eqref{goal3} holds.

First, it is easy to see that \eqref{goal3} holds when $n=3$ (cf. \cite{TW2} for the same argument in a related context):
\[\begin{split}
\int_X(\beta+\ve\omega+\ddbar u_\ve)^3&=\int_X(\beta+\ddbar u_\ve)^3+3\ve\int_X(\beta+\ve\omega+\ddbar u_\ve)^2\wedge\omega\\
&-3\ve^2\int_X(\beta+\ve\omega+\ddbar u_\ve)\wedge \omega^2+\ve^3\int_X\omega^3\\
&\geq \int_X\alpha^3-3\ve^2\int_X(\beta+\ve\omega+\ddbar u_\ve)\wedge \omega^2,
\end{split}\]
and if we pick $\omega$ Gauduchon then the last term equals $-3\ve^2\int_X(\beta+\ve\omega)\wedge \omega^2$ which goes to zero as $\ve\to 0$, proving \eqref{goal3}.

Second, for general dimension $n$, we observe that to prove \eqref{goal3} it would be enough to produce smooth functions $\ti{h}_\ve$ such that
$\alpha+\ve\omega+\ddbar \ti{h}_\ve>0$, and so that
\begin{equation}\label{todo2}
\ve^\delta\|\ti{h}_\ve\|_{L^\infty(X)}\to 0,
\end{equation}
as $\ve\to 0$, for some $0<\delta<1/(n-2).$ Indeed, from the definition of the envelope $u_\ve$ we obtain
$$\ti{h}_\ve-u-\sup_X(\ti{h}_\ve-u)\leq u_\ve\leq 0,$$
hence $\|u_\ve\|_{L^\infty(X)}\leq \|\ti{h}_\ve\|_{L^\infty(X)}+C,$ and so \eqref{todo2} implies \eqref{todo}.

For example, one could try to use the solutions of
$$(\alpha+\ve\omega+\ddbar \ti{h}_\ve)^n=e^{\ti{h}_\ve}\omega^n,$$
which exist by Cherrier, \cite{Ch}, and
which in the special case when $\alpha\geq 0$ satisfy
$$n\log\ve\leq \ti{h}_\ve\leq C$$
by the maximum principle, and so satisfy \eqref{todo2}. Or one could try to use the solutions of
$$(\alpha+\ve\omega+\ddbar \ti{h}_\ve)^n=C_\ve \omega^n,$$
which are given by Tosatti-Weinkove \cite{TW}.
However, in general it remains unclear to us whether the functions $\ti{h}_\ve$ produced by either method can be proved to satisfy \eqref{todo2}.
\end{remark}

\section{An application}\label{appl}
We now recall an application of Conjecture \ref{con} to complex Monge-Amp\`ere equations on non-K\"ahler manifolds, taken from \cite{To}.

Let $(X^n,\omega)$ be a compact Hermitian manifold with a closed real $(1,1)$ form $\alpha$ with $[\alpha]$ nef and $\int_X\alpha>0$. By assumption, for every $\ve>0$ there is $h_\ve\in C^\infty(X,\mathbb{R})$ such that $\alpha+\ve\omega+\ddbar h_\ve>0$. Suppose that for all $\ve>0$ we are also given a smooth positive volume form $\Omega_\ve$ with $\int_X\Omega_\ve=1$. Thanks to \cite{TW} we can find $\vp_\ve\in C^\infty(X,\mathbb{R})$ smooth functions solving
$$\alpha+\ve\omega+\ddbar (h_\ve+\vp_\ve)>0,\quad (\alpha+\ve\omega+\ddbar (h_\ve+\vp_\ve))^n=C_\ve\Omega_\ve,$$
for some (uniquely determined) positive constants $C_\ve$. In the K\"ahler case of course we have that
$$C_\ve=\int_X(\alpha+\ve\omega)^n\geq\int_X\alpha^n.$$
In the non-K\"ahler case, it is important to find a uniform positive lower bound for $C_\ve$ (see e.g. \cite{To,TW2}). This can be achieved using Conjecture \ref{con}, as observed in \cite[Remark 3.3]{To}:
\begin{proposition}\label{bd}
If Conjecture \ref{con} holds, then we have
\begin{equation}\label{bdd}
C_\ve\geq \int_X\alpha^n.
\end{equation}
\end{proposition}
\begin{proof}
Indeed, we take $\beta_\ve=\alpha+\ddbar(h_\ve+\vp_\ve)$. Thanks to Conjecture \ref{con} we have
\[\begin{split}
\int_X \alpha^n&\leq \int_{X(\beta_\ve,0)}(\alpha+\ddbar(h_\ve+\vp_\ve))^n\\
&\leq \int_{X(\beta_\ve,0)}\left(\alpha+\ve\omega+\ddbar(h_\ve+\vp_\ve)\right)^n\\
&=C_\ve\int_{X(\beta_\ve,0)}\Omega_\ve\leq C_\ve\int_X\Omega_\ve=C_\ve,
\end{split}
\]
as required.
\end{proof}
Of course, we only need the ``half'' \eqref{g3} of Conjecture \ref{con}. Thanks to Theorem \ref{main}, Proposition \ref{1} and Remark \ref{rmk} we see that \eqref{bdd} holds when $n\leq 3$, or $[\alpha]$ is semipositive, or $X$ is in class $\mathcal{C}$.


\begin{thebibliography}{99}
\bibitem{Be2} Berman, R.J. {\em Bergman kernels and equilibrium measures for line bundles over projective manifolds}, Amer. J. Math. {\bf 131} (2009), no. 5, 1485--1524.
\bibitem{Be} Berman, R.J. {\em From Monge-Amp\`ere equations to envelopes and geodesic rays in the zero temperature limit},  Math. Z. {\bf 291} (2019), no. 1-2, 365--394.
\bibitem{BD} Berman, R. and Demailly, J.-P., {\em Regularity of plurisubharmonic upper envelopes in big cohomology classes}, in {\em Perspectives in analysis, geometry, and topology}, 39--66, Progr. Math., 296, Birkh\"auser/Springer, New York, 2012.
\bibitem{Bo} Boucksom, S. {\em On the volume of a line bundle}, Internat. J. Math. {\bf 13} (2002), no. 10, 1043--1063.
\bibitem{BDPP} Boucksom, S., Demailly, J.-P., P\u{a}un, M., Peternell, T. {\em The pseudo-effective cone of a compact K\"ahler manifold and varieties of negative Kodaira dimension}, J. Algebraic Geom. {\bf 22} (2013), no. 2, 201--248.
\bibitem{Ch} Cherrier, P. {\em \'Equations de Monge-Amp\`ere sur les vari\'et\'es hermitiennes compactes}, Bull. Sci. Math. (2) {\bf 111} (1987), no. 4, 343--385.
\bibitem{CTW} Chu, J., Tosatti, V., Weinkove, B. {\em The Monge-Amp\`ere equation for non-integrable almost complex structures}, J. Eur. Math. Soc. (JEMS) {\bf 21} (2019), no.7, 1949--1984.
\bibitem{CZ} Chu, J., Zhou, B. {\em Optimal regularity of plurisubharmonic envelopes on compact Hermitian manifolds}, Sci. China Math. {\bf 62} (2019), no. 2, 371--380.
\bibitem{De3} Demailly, J.-P. {\em Champs magn\'etiques et in\'egalit\'es de Morse pour la $d''$-cohomologie}, Ann. Inst. Fourier (Grenoble) {\bf 35} (1985), no. 4, 185--229.
\bibitem{De} Demailly, J.-P. {\em Holomorphic Morse inequalities and asymptotic cohomology groups: a tribute to Bernhard Riemann}, Milan J. Math. {\bf 78} (2010), no. 1, 265--277.
\bibitem{De2} Demailly, J.-P. {\em A converse to the Andreotti-Grauert theorem}, Ann. Fac. Sci. Toulouse Math. (6) {\bf 20} (2011), Fascicule Sp\'ecial, 123--135.
\bibitem{DP} Demailly, J.-P., P\u{a}un, M. \emph{Numerical characterization of the K\"ahler cone of a compact K\"ahler manifold}, Ann. of Math., {\bf 159} (2004), no. 3, 1247--1274.
\bibitem{DK} Dinew, S., Ko\l odziej, S. {\em Pluripotential estimates on compact Hermitian manifolds}, in {\em Advances in geometric analysis}, 69--86, Adv. Lect. Math. (ALM), 21, Int. Press, Somerville, MA, 2012.
\bibitem{KN} Ko\l odziej, S., Nguyen, N.C. {\em Weak solutions to the complex Monge-Amp\`ere equation on Hermitian manifolds}, in {\em  Analysis, complex geometry, and mathematical physics: in honor of Duong H. Phong}, 141--158, Contemp. Math., 644, Amer. Math. Soc., Providence, RI, 2015.
\bibitem{KN1} Ko\l odziej, S., Nguyen, N.C. {\em Stability and regularity of solutions of the Monge–Ampère equation on Hermitian manifolds}, Adv. Math. {\bf 346} (2019), 264--304.
\bibitem{Ng} Nguyen, N.C. {\em The complex Monge-Amp\`ere type equation on compact Hermitian manifolds and applications}, Adv. Math. {\bf 286} (2016), 240--285.
\bibitem{PS} Phong, D.H., Sturm, J. {\em The Dirichlet problem for degenerate complex Monge-Amp\`ere equations}, Comm. Anal. Geom. {\bf 18} (2010), no. 1, 145--170.
\bibitem{PS2} Phong, D.H., Sturm, J. {\em Regularity of geodesic rays and Monge-Amp\`ere equations}, Proc. Amer. Math. Soc. {\bf 138} (2010), no. 10, 3637--3650.\
\bibitem{PS3} Phong, D.H., Sturm, J. {\em On the singularities of the pluricomplex Green's function}, in {\em Advances in analysis: the legacy of Elias M. Stein}, 419--435, Princeton Math. Ser., 50, Princeton Univ. Press, Princeton, NJ, 2014.
\bibitem{Si} Siu, Y.-T. {\em Calculus inequalities derived from holomorphic Morse inequalities}, Math. Ann. {\bf 286} (1990), no. 1-3, 549--558.
\bibitem{To} Tosatti, V. {\em The Calabi-Yau Theorem and K\"ahler currents},  Adv. Theor. Math. Phys. {\bf 20} (2016), no. 2, 381--404.
\bibitem{To2} Tosatti, V. {\em Regularity of envelopes in K\"ahler classes}, Math. Res. Lett. {\bf 25} (2018), no.1, 281--289.
\bibitem{TW} Tosatti, V., Weinkove, B. {\em The complex Monge-Amp\`ere equation on compact Hermitian manifolds}, J. Amer. Math. Soc. {\bf 23} (2010), no. 4, 1187--1195.
\bibitem{TW2} Tosatti, V., Weinkove, B. {\em Plurisubharmonic functions and nef classes on complex manifolds}, Proc. Amer. Math. Soc. {\bf 140} (2012), no. 11, 4003--4010.
\bibitem{WN} Witt Nystr\"om, D. {\em Duality between the pseudoeffective and the movable cone on a projective manifold. With an appendix by S\'ebastien Boucksom}, to appear in J. Amer. Math. Soc.
\end{thebibliography}
\end{document}